\documentclass{amsart}
\usepackage{amsmath, amssymb, amsthm, amscd, amsfonts, eucal, hyperref}
\usepackage[all]{xy}

\newtheorem{theorem}{Theorem}[section]
\newtheorem{proposition}[theorem]{Proposition}
\newtheorem{lemma}[theorem]{Lemma}

\theoremstyle {definition}
\newtheorem {definition}[theorem]{Definition}
\newtheorem {example}[theorem]{Example}
\theoremstyle {remark}
\newtheorem{remark}[theorem]{Remark}

\def\rank{\operatorname{rank}}

\def\ord{\operatorname{ord}}
\def\grad{\operatorname{grad}}

\def\supp{\mathrm{supp}}

\keywords{global Milnor fibration, bifurcation value, atypical value, Malgrange condition, M-tame, Newton polyhedron, non-degenerate, One dimensional fiber, trivial fibration. \\}
\begin{document}

\title[Bifurcation set]{Bifurcation set, M-tameness, asymptotic critical values and Newton polyhedrons}
\author{Nguyen Tat Thang}

\address{Institute of Mathematics, 18 Hoang Quoc Viet road, 10307 Hanoi, Vietnam.}

\email{ntthang@math.ac.vn}


\begin{abstract}
Let $F=(F_1, F_2, \ldots, F_m): \mathbb{C}^n \to \mathbb{C}^m$ be a polynomial dominant mapping  with $n>m$.
In this paper we give the relations between the bifurcation set of $F$ and the set of values where $F$ is not M-tame as well as the set of generalized critical values of $F$. We also construct explicitly a proper subset  of $\mathbb{C}^m$ in terms of the Newton polyhedrons of $F_1, F_2, \ldots, F_m$ and show that it contains the bifurcation set of $F$. In the case $m= n-1$ we show that $F$ is a locally $C^{\infty}$-trivial fibration if and only if it is a locally $C^0$-trivial fibration.
 \end{abstract}
\maketitle

 \section{Introduction}
Let $F=(F_1, F_2, \ldots, F_m): \mathbb{C}^n \to \mathbb{C}^m$ be a polynomial dominant mapping  with $n>m$. It is
well-known that the mapping $F$ is a locally $C^{\infty}$-trivial fibration outside a {\it bifurcation
set} $B(F)$ (see \cite{T}). In general, the set $B(F)$ is larger than $K_0(F)$ --- the set of critical
values of $F$. It contains also the set $B_{\infty}(F)$ of {\it critical values at infinity}.
Roughly speaking, the set $B_{\infty}(F)$ consists of points at which $F$ is not a locally
$C^{\infty}$-trivial fibration at infinity (i.e., outside a compact set).

It is a natural question to ask how the set $B(F)$ can be computed. The answer was given for polynomial functions in two variables (see, for example \cite{HL}, \cite{H}, \cite{I}) where the bifurcation set is determined in terms of the topological properties of the fibers (Euler characteristic, transversal crossing with balls, ...) and for polynomials which have only isolated singularities at infinity (\cite{P}).
 
 The aim of this paper is to generalize the results in \cite{Nemethi1990} where the author proved that the bifurcation set of polynomial functions is contained in some explicit subsets of $\mathbb{C}$  for the case of polynomial maps from $\Bbb{C}^n$ to $\Bbb{C}^m$. 

In order to state the main theorems, let us introduce the following notion.
 
\begin{definition}{\rm 
Let $t\in \mathbb{C}^m$ arbitrary. The map $F$ is called to be {\it M-tame at $t$} if there does not exist a sequence $\{p_k\}_{k=1}^{\infty}$ such that
$$\|p_k\|\to \infty, F(p_k)\to t\, \, \textrm{and}\, \, \rank \begin{pmatrix}
J(F)(p_k)\\
 \overline{p_k}
\end{pmatrix}\leq m,$$
where $J(F)$ is the Jacobi matrix of $F$. We denote by $M_{\infty}(F)$ the set of $t\in \mathbb{C}^m$ at which $F$ is not M-tame. Let $M(F):= K_0(F)\cup M_{\infty}(F)$.
}
\end{definition}
When $m=1$ the notion of M-tame was introduced by A. Nemethi and A. Zaharia (see \cite{Nemethi1990} and \cite{Nemethi1992}).

 The function $\nu$ from set of linear maps $A: \Bbb{C}^n\to \Bbb{C}^m$ to the complex numbers is defined as follows (see \cite{R}):
 $$\nu(A)=\inf_{\{\omega \in \mathbb{C}^m: \|\omega \|= 1\}}\|A\omega \|.$$
The set of {\it asymptotic critical values at infinity} of $F$ is defined by (see \cite{KOS} and \cite{R})
\begin{eqnarray*}
K_\infty(F) := \{ t \in {\Bbb C}^m
& | & \textrm{there exists a sequence $x_l \rightarrow \infty$ such that}\\
&& F(x_l) \rightarrow t \ \textrm{and} \ \|x_l\|\cdot\nu (\text{d} F(x_l))\to 0\}.
\end{eqnarray*}
Let $K(F): = K_0(F) \cup K_{\infty}(F)$. We say that $K(F)$ is the set of {\it generalized
critical values} of $F$. 

It is proven in \cite{Nemethi1990} that
\begin{theorem}{\rm(See \cite{Nemethi1990})}
Let $f: \mathbb{C}^n \to \mathbb{C}$ be a polynomial function. Then
$$B(f)\subseteq  M(f).$$
\end{theorem}

In Section 2 of this paper, we prove the following theorem.

\begin{theorem}\label{theoremBifurcationM-tame}
Let $F: \mathbb{C}^n \to \mathbb{C}^m$ be a polynomial map. Then
$$B(F)\subseteq  M(F)\subseteq  K(F).$$
\end{theorem}

 In Section 3, we introduce the notion of {\it Newton non-degenerate} polynomial mappings and a proper subset $\Sigma(F)$ of $\mathbb{C}^m$ which is constructed explicitly in terms of the Newton polyhedrons of coordinate polynomials $F_1, F_2, \ldots, F_m$. Our second main theorem (Theorem \ref{theoremBifurcationNewton}) shows that if the polynomial map $F$ is Newton non-degenerate  then $B(F)\subseteq \Sigma (F)$. This is an extension of Theorem 2 in \cite{Nemethi1990}.

In the last section, we prove that for polynomial maps $F: \Bbb{C}^n\to \Bbb{C}^{n-1}$ the problem of computing bifurcation values is, in some sense, a topological problem. More precisely, we show that if $F$ is a locally $C^{0}$-trivial fibration then it is a locally $C^{\infty}$-trivial fibration (Theorem \ref{mainthmst4}).

\section{M-tameness and generalized critical values at infinity}
In this section, we will prove that the bifurcation set of a polynomial map is contained in the set of values at which the map is not M-tame, as well as the set of generalized critical values of the map.

\begin{proof}[Proof of Theorem \ref{theoremBifurcationM-tame}]

Let $t^0=(t^0_1, t^0_2, \ldots, t^0_m)$ be a regular value of $F$ such that $t^0\notin M_{\infty}(F)$. We will show that $F$ defines a trivial fibration in some neighbourhood of $t^0$. 

Indeed,  since $F$ is M-tame at $t^0$ there exist a closed ball $B$, centered at the origin, in $\mathbb{C}^n$ and a neighbourhood $D$ of $t^0$ such that
$$
\rank \begin{pmatrix}
J(F)(x)\\
 \bar{x}
\end{pmatrix}> m,\, \, x\in F^{-1}(D)\setminus B.$$
 Hence, for each $i=1,\ldots, m$, we can construct in $F^{-1}(D)\setminus B$ a smooth vector field $v^i$ such that
$$
\begin{aligned}
\langle & v^i, x\rangle = 0\\
\langle& v^i, \grad F_i\rangle= 1\\
\langle& v^i, \grad F_j\rangle= 0,\, \, j\neq i,
\end{aligned}
$$
where by $\mathrm{grad} \varphi$ of a function $\varphi$ we mean the vector $(\frac{\overline{\partial \varphi}}{\partial x_1},\ldots, \frac{\overline{\partial \varphi}}{\partial x_n})$.
By integrating those vector fields we obtain a diffeomorphism trivializing the restriction 
$$F_{|F^{-1}(D)}: F^{-1}(D)\to D.$$
Thus
$$B(F)\subseteq  M_{\infty}(F)\cup K_{0}(F)= M(F).$$

Now, let $t^0\in M_{\infty}(F)$ be an arbitrary regular value of $F$. To complete the proof, it suffices to show that $t^0\in K_{\infty}(F)$. Because $F$ is not M-tame at $t^0$, there exists a sequence $\{p_k\}_{k=1}^{\infty}\subset \mathbb{C}^n$ such that
$$\|p_k\|\to \infty, F(p_k)\to t^0\, \, \textrm{and}\, \, \rank \begin{pmatrix}
J(F)(p_k)\\
 \overline{p_k}
\end{pmatrix}\leq m.$$
Then, since  $t^0$ is a regular value of $F$ there is a sequence $\{s_k\}_{k=1}^{\infty}\subset \mathbb{C}^m$ satisfying
$$p_k= \sum_{i=1}^m s_{k, i}\cdot \grad F_i(p_k),$$
where $s_k= (s_{k, 1}, s_{k, 2}, \ldots, s_{k, m}).$ Therefore, according to the Curve Selection Lemma (see \cite{Milnor1968}, \cite{Nemethi1992}), there exist analytic curves $\varphi(s)$ in $\mathbb{C}^n$ and  $\lambda(s)=(\lambda_1(s), \lambda_2(s), \ldots, \lambda_m(s))$ in $\mathbb{C}^m$ such that
\begin{itemize}
\item [(a1)] $\lim_{s \to 0}\|\varphi(s)\|= \infty,$
\item [(a2)] $\lim_{s \to 0}F(\varphi(s))= t^0,$
\item [(a3)] $\varphi(s)= \sum_{i=1}^m\lambda_i(s) \grad F_i(\varphi(s)).$
\end{itemize}

We have
$$\nu(\text{d}F(\varphi))= \min_{\|\omega\|=1}\left\|\sum_{i=1}^m\omega_i\grad F_i(\varphi)\right\|.$$
So
\begin{align}\frac{\|\varphi\|^2}{\|\lambda\|}=\|\varphi\|\cdot \left\|\sum_{i=1}^m\frac{\lambda_i}{\|\lambda\|}\cdot \grad F_i(\varphi)\right\|\geq \|\varphi\|\cdot \nu(\text{d}F(\varphi)).\label{bdt3.2}
\end{align}

On the other hand, we have the following estimation
\begin{align}
\left|\frac{d\|\varphi\|^2}{2ds}\right| =| Re\langle \varphi, \varphi^{'}\rangle| &\leq |\langle \varphi, \varphi^{'}\rangle| \notag\\
                                           &= \left|\left\langle\sum_{i=1}^m\lambda_i \grad F_i(\varphi), \varphi^{'}\right\rangle\right|\notag\\
                                           &=\left|\left\langle\lambda, \frac{dF(\varphi)}{ds}\right\rangle\right| \leq \|\lambda\|\cdot \left\|\frac{dF(\varphi)}{ds}\right\|.\notag
\end{align}
If $F(\varphi(s))\equiv t^0$ then $\|\varphi(s)\|$ is constant, which contradicts condition (a1). Hence condition (a2) implies that we may express $F(\varphi(s))$ as follows:
$$F(\varphi(s))= t^0 + cs^{\rho} + \textit{terms of higher exponents},$$
where $c\in \mathbb{C}^m\setminus\{(0, \ldots, 0)\}$ and $\rho>0.$ In particular, we get
$$\ord(\|\frac{dF(\varphi)}{ds}\|)= \rho - 1.$$
Hence
\begin{align}
\ord(\|\varphi\|^2) - \ord(\|\lambda\|)\geq \rho >0.\label{bdt3.3}
\end{align}

It follows from (\ref{bdt3.2}) and (\ref{bdt3.3}) that
 $$\lim_{s \to 0}\|\varphi\|\cdot \nu(\text{d}F(\varphi)) = 0.$$
Combining this with (a1) and (a2) we obtain $t^0\in K_{\infty}(F)$.
\end{proof}

\begin{remark}
{\rm (i) Theorem \ref{theoremBifurcationM-tame} was proved in \cite{CT} for mixed functions and in \cite{DRT}  for
  real maps.

(ii) It follows from Theorem \ref{theoremBifurcationM-tame} that $B(F)\subseteq K(F)$. This fact was proved in  \cite{R}, \cite{G} and \cite{J}. However, the equality does not occur in general, see \cite{HT}, Proposition 3.2. The following example show that the M-tameness is indeed better in controlling the topology of the map at infinity than the generalized critical values.
}
\end{remark}

\begin{example}\label{ex1}
Let $$F=(xy-1, y^2z): \Bbb{C}^3\to \Bbb{C}^2.$$
It is easy to check that $K_0(F)= \{(-1, 0)\}$ and $B(F)= K_0(F)$, (see \cite{HT}). 

To compute $M_{\infty}(F)$, we see that
\begin{align}\label{dk1}\rank \begin{pmatrix}
J(F)(p)\\
 \overline{p}
\end{pmatrix}\leq 2\end{align}
 if and only if $y=0$ or $x\bar{x} - y \bar{y} + 2z \bar{z}=0,$ where $p = (x, y, z)$. The first case implies $(-1, 0)\in M_{\infty}(F).$ In the second case, let $(t_1, t_2):= F(p)$, one can assume that $y\neq 0$, we obtain
$$x= \frac{t_1+1}{y}, z= \frac{t_2}{y^2},$$
where $(y\bar{y})^3 - (t_1+1)(\bar{t_1}+1)y\bar{y} - 2t_2\bar{t_2}=0.$ Then, we can easily check that there exists a sequence of point $(x, y, z)$ going to infinity and satisfying (\ref{dk1}) only when $(t_1, t_2)$ tends to $(-1, 0).$ 

Thus $M_{\infty}=\{(-1, 0)\}$. It means $B(F)= K_0(F)\cup M_{\infty}(F).$ 
Nevertheless, it was shown in \cite{HT} that the set of asymptotic critical values $K_{\infty}(F)$ contains at least $(0, 0)\notin B(F).$
\end{example}

We know from \cite{R, J} that the set $K(F)$ of generalized critical values of a polynomial map $F$ is an algebraic set, but it is not easy to find a defining equation for such a set. In the later example, we give a defining equation for $B(F), M(F)$ and $K(F)$ where they are all equal. First of all, we recall the notion of {\it Gaffney number} which can be used to determine the asymptotic critical values.

\begin{definition}[See \cite{G}]{\rm
Let $A: \Bbb{C}^n\to \Bbb{C}^m$ be a linear mapping ($n\geqslant m$). Let $a=[a_{ij}]$ be the matrix of $A$. Let $M_I$, where $I=(i_1, \ldots, i_m)$, denote an $(m\times m)$ minor of $a$ given by columns indexed by $I$. Let $M_J(j)$ denote an $(m-1\times m-1)$ minor given by columns indexed by $J$ and by deleting the $j$th row (if $m=1$ we put $M_J(j)=1$). Then by the {\it Gaffney number of $A$} we mean the number
$$g(A)=\frac{(\sum_I|M_I|^2)^{1/2}}{(\sum_{J,j}|M_J(j)|^2)^{1/2}}.$$
(If this number is not defined we put $g(A)=0.$)
}
\end{definition}

It deduces from \cite[Propositions 2.2 and 2.3]{J}  or \cite[Remark 4.1]{KOS} that
\begin{proposition} \label{md2.4}
Let $F: \mathbb{C}^n \to \mathbb{C}^m $ be a polynomial mapping. Then
\begin{eqnarray*}
K_\infty(F) := \{ t \in {\Bbb C}^m
& \mid & \textrm{there exists a sequence $x_l \rightarrow \infty$ such that}\\
&& F(x_l) \rightarrow t \ \mathrm{and} \ \|x_l\|\cdot g (\text{d} F(x_l))\to 0\}.
\end{eqnarray*}
\end{proposition}

\begin{example}\label{ex2}
Let  $$F=(xy+1,(xyz+1)(xyz+z-1)): \Bbb{C}^3\to \Bbb{C}^2.$$ 
For $t_1\neq 1$ we have
$$F^{-1}(t_1,t_2)=\{((t_1-1)/y,y,z_i)\ | \ y\in \Bbb{C}^{*},\  i=1, 2\},$$ 
where $z_i$ satisfies
\begin{align}
t_1(t_1-1)z_i^2+z_i-(1+t_2)=0,\label{ptvd2}
\end{align}
see \cite{HT}. Note that if $t_1(t_1-1)\to 0$ then either $z_1$ or $z_2$ goes to infinity.

One can show that $$K_0(F)=\{(u, v)\mid u=1, v\in \mathbb{C}\}\cup \{(u, v)\mid 4u(u-1)(v+1)+1=0\}$$
and
$$B(F)= \{(u, v)\mid u(u-1)=0\}\cup \{(u, v)\mid 4u(u-1)(v+1)+1=0\}.$$

Let consider $(a, b)\in K_{\infty}(F)$. By the definition of asymptotic critical values and Proposition \ref{md2.4}, there exists a sequence $\{(x_k, y_k, z_k)\}$ going to infinity such that
$$(t_1^k, t_2^k):= F(x_k, y_k, z_k)\to (a, b)$$
and
$$\|(x_k, y_k, z_k)\|\cdot g (\text{d} F(x_k, y_k, z_k))\to 0.$$
 By a computation we obtain
$$g (\text{d} F(x_k, y_k, z_k))^2= \frac{|2t_1^k(t_1^k-1)z_k+1|^2 (|x_k|^2+ |y_k|^2)}{|x_k|^2+ |y_k|^2+ |2t_1^k(t_1^k-1)z_k+1|^2+ |z_k|^4(|x_k|^2+ |y_k|^2) |2t_1^k-1|^2}.$$
Put $A_k= |2t_1^k(t_1^k-1)z_k+1|^2$ and $B_k= |x_k|^2+ |y_k|^2.$ There are three cases.

{\bf Case 1}: $A_k\to \infty.$ It implies from (\ref{ptvd2}) that 
$$A_k= \left|\frac{2(t_2^k+1)}{z_k}-1\right|^2.$$
 Therefore $z_k\to 0$, and then $B_k= |x_k|^2+ |y_k|^2\to \infty.$ This follows that
$$\|(x_k, y_k, z_k)\|^2\cdot g (\text{d} F(x_k, y_k, z_k))^2= \frac{\|(x_k, y_k, z_k)\|^2}{\frac{1}{A_k}+ \frac{1}{B_k}+ \frac{|z_k|^4\cdot |2t_1^k-1|^2}{A_k}}\to \infty,$$
a contradiction.

{\bf Case 2}: $A_k\not\to \infty$ and $z_k\to \infty$. It is obvious that $t_1^k(t_1^k+1)\to 0$. In particular $(a, b)\in B(F).$

{\bf Case 3}: $A_k\not\to \infty$ and $z_k\not\to \infty$. In this case, we see that $B_k\to \infty$ and the sequence $\{|z_k|\}$ is bounded from above by some positive constant $c$. We have
$$\|(x_k, y_k, z_k)\|^2\cdot g (\text{d} F(x_k, y_k, z_k))^2\geq  \frac{\|(x_k, y_k, z_k)\|^2\cdot A_k}{1+ \frac{A_k}{B_k}+ |c|^4\cdot |2t_1^k-1|^2}.$$
Since $\|(x_k, y_k, z_k)\|\cdot g (\text{d} F(x_k, y_k, z_k))\to 0$ and $(x_k, y_k, z_k)\to \infty$ as $k\to \infty$ we get $A_k\to 0$. Combining this with the fact $A_k= |\frac{2(t_2^k+1)}{z}-1|^2$ gives us $z_k\to 2(b+1)$. Therefore
$$A_k= |2t_1^k(t_1^k-1)z_k+1|^2 \to |4a(a-1)(b+1)+1|^2.$$
That means $4a(a-1)(b+1)+1=0$, in other words $(a, b)\in B(F)$.

Thus, we have shown that $K(F)\subseteq B(F)$. Moreover, according to Theorem \ref{theoremBifurcationM-tame} that
$$B(F)\subseteq M(F) \subseteq K(F).$$
 Then
$$B(F)= M(F)= K(F).$$
\end{example}

\section{Newton polyhedrons}
In this section, we will generalize the notion of Newton non-degeneracy of polynomial functions for polynomial maps and give an estimation for the bifurcation set $B(F)$ of $F$ in terms of $F_i$'s Newton polyhedrons.

Firstly, let us recall some notations and definitions, see \cite{Kouchnirenko1976}, \cite{Nemethi1990}.
Let $f \colon {\Bbb C}^n
\rightarrow {\Bbb C}$ be a polynomial function. We express $f$ as follows:
$$f(z) := \sum_{\alpha \in {\Bbb Z}^n_{\geq 0}} a_\alpha z^{\alpha},$$
where ${\Bbb Z}_{\geq 0}$ denotes the set of non-negative integers. The {\sl support}
$\supp(f)$ is defined to be $\{ \alpha \ | \ a_\alpha \ne 0\}.$ We denote
$\Gamma_-(f)$ to be the convex hull of the set $\{0\} \cup \textrm{supp}(f).$
The {\sl Newton boundary at infinity} $\Gamma_{\infty}(f)$ is by definition the
union of the closed faces of the polyhedron $\Gamma_-(f)$ which do not contain
the origin. Here and below, by face we shall understand face of {\it any}
dimension. For each closed face $\Delta$ of $\Gamma_{\infty}(f)$ we denote by
$f_{ \Delta}$ the polynomial $\sum_{\alpha \in \Delta} a_\alpha z^{\alpha}.$ 

The polynomial $f$ is called {\sl Newton non-degenerate} if for each face
$\Delta \in \Gamma_{\infty}(f),$ the system of equations
$$\frac{\partial f_\Delta}{\partial z_1} = \frac{\partial f_\Delta}{\partial z_2} = \cdots = \frac{\partial f_\Delta}{\partial z_n}
= 0$$ has no solutions in $({\Bbb C}^{*})^n.$ The polynomial $f$ is called
{\em convenient} if the intersection of $\supp(f)$ with each coordinate axis is
non-empty.

Let
$\overline{\textrm{supp}(f)}$ be the convex hull in $\mathbb{R}^n$ of the set $\supp(f).$ A closed face $\Delta$ of $\overline{\supp(f)}$ is called {\em  bad}
if
\begin{enumerate}
  \item [(i)] the affine subvariety of dimension = $\dim (\Delta)$ spanned by $\Delta$ contains the origin, and
  \item [(ii)] there exists a hyperplane $H \subset {\Bbb R}^n$ with equation $a_1\alpha_1 + a_2 \alpha_2 + \cdots + a_n \alpha_n = 0,$
where $\alpha_1, \alpha_2, \ldots, \alpha_n$ are the coordinates in ${\Bbb R}^n,$ such that
\begin{enumerate}
  \item [(ii$_a$)] there exist $i$ and $j$ with $a_i \cdot a_j < 0,$ and
  \item [(ii$_b$)] $H \cap \overline{\supp(f)} = \Delta.$
\end{enumerate}
\end{enumerate}
More geometrically, the condition (ii$_a$) says that the hyperplane $H$ intersects the region of ${\Bbb R}^n$ whose coordinates are all positive. We denote by $\mathcal B$ the set of bad faces of
$\overline{\supp(f)}.$ For $\Delta \in \mathcal B,$ we define
$${\Sigma}^{'}_0(f_\Delta) := \{ f_\Delta(z^0) \ | \ z^0 \in ({\Bbb C}^{*})^n \textrm{ and } \grad f_\Delta(z^0) = 0 \}.$$
Let 
$$\Sigma_\infty(f) := \cup_{\Delta \in {\mathcal B}} {\Sigma}^{'}_0(f_\Delta).$$
It is clear that ${\Sigma}^{'}_0(f_\Delta) \subset K_0(f_\Delta).$ This,
together with an algebraic version of Sard's theorem (see \cite{Benedetti1990}),
yields that $\Sigma_\infty(f)$ is a finite set.

The following result gives an estimation for the bifurcation set $B(f)$ of $f$ in terms of
its Newton boundary at infinity.
\begin{theorem} \label{TheoremBroughton1988}
{\rm (\cite{Kouchnirenko1976}, \cite{Broughton1988}, \cite{Nemethi1990}) } Let $f \colon {\Bbb C}^n
\rightarrow  {\Bbb C}$ be a Newton non-degenerate polynomial function. Then the following
statements hold:
\begin{enumerate}
  \item [(i)] If $f$ is convenient, then $B(f) = K_0(f).$
  \item [(ii)] If $f$ is not convenient, then $B(f) \subseteq K_0(f) \cup \Sigma_\infty(f) \cup \{f(0) \}.$
\end{enumerate}
\end{theorem}

Now, let  $F= (F_1, F_2, \ldots, F_m): \Bbb{C}^n \to  \Bbb{C}^m$ be a polynomial mapping. Assume that $F_i(0, 0, \ldots, 0)=0$. In the following definition, we generalize the notion of Newton non-degeneracy for polynomial maps (see \cite{Oka} for the local case). 

\begin{definition}{\rm We denote by $\Gamma_{-}(F)$ the following set:
$$\Gamma_{-}(F):= \Gamma_{-}(F_1)\times \cdots \times \Gamma_{-}(F_m).$$
The dual space of $\Bbb{R}^n$ can be canonically identified with $\Bbb{R}^n$ itself by the Euclidean inner product. By a {\it covector} we mean the integral dual vector. We use column vectors to show the dual vectors. For a given covector $P= (p_1, \ldots, p_n)^t$, for each $\beta= (\beta_1, \ldots, \beta_n)\in ~\Bbb{R}^n$, $P(\beta)$ is defined by $\beta P= \sum_{i=1}^n\beta_ip_i$. Denote $\Delta(P; F)= ~(\Delta(P; F_1), \ldots, \Delta(P; F_m)),$ where $\Delta(P; F_i)$ is the face of the polyhedron $\Gamma_{-}(F_i)$ on which the restriction $P|_{\Gamma_{-}(F_i)}$ attains its minimum. We define $\Gamma_{\infty}(F)$ to be the set of $m$-tuples $\Delta(P; F)$ in $\Gamma_{-}(F)$ with $(0, \ldots, 0)\notin \Delta(P; F)$.

The map $F$ is called {\em Newton non-degenerate} (or {\it non-degenerate} for short) if
$$\{a : \rank(J((F_1)_{\Delta_1}, (F_2)_{\Delta_2}, \ldots, (F_m)_{\Delta_m})(a) < m\} \cap (\Bbb{C}^*)^n= \emptyset$$
for any $m$-tuples $(\Delta_1, \Delta_2, \ldots, \Delta_m)$ in $\Gamma_{\infty}(F)$.

For each $\Delta= (\Delta_1, \Delta_2, \ldots, \Delta_m)$, where $\Delta_i\in \mathcal B(F_i)$  for all $i=1, \ldots, m$, set
\begin{align*}
{\Sigma}^{'}_0(F_{\Delta}) := \{ &({(F_1)}_{\Delta_1}(z^0), {(F_2)}_{\Delta_2}(z^0), \ldots, {(F_m)}_{\Delta_m}(z^0)) \ \mid \ z^0 \in ({\Bbb C}^*)^n \, \textrm{and}\\
&  \rank (J((F_1)_{\Delta_1}, (F_2)_{\Delta_2}, \ldots, (F_m)_{\Delta_m})(z^0)) < m \}.
\end{align*}
Put
$$\Sigma_\infty(F) := \cup_{\Delta \in {\mathcal B(F_1)}\times {\mathcal B(F_2)}\times \cdots \times {\mathcal B(F_m)}} {\Sigma}^{'}_0(F_\Delta).$$
One notes that $\Sigma_\infty(F)\subset \Bbb{C}^m$ is a semi-algebraic subset of dimension less than $m$.
}\end{definition}

\begin{theorem}\label{theoremM-tameNewton}
Let $F = (F_1, F_2, \ldots, F_m): \Bbb{C}^n \to  \Bbb{C}^m$ be a non-degenerate polynomial map. Then
$$M_{\infty}(F) \subseteq \Sigma_\infty(F) \cup \bigcup_{i=1}^m \{t=(t_1, t_2, \ldots, t_m) \in \Bbb{C}^m : t_i = F_i (0, 0, \ldots, 0)\}.$$
\end{theorem}

\begin{proof}
Without loss of generality, we may assume that $F_i (0, 0, \ldots, 0)= 0 $ for all $i=1, \ldots, m$. Let $t^0= (t^0_1, \ldots, t^0_m) \in M_{\infty}(F)$ such that
$t^0_i \neq 0$ for all $i=1, \ldots, m$. We need to prove that $t^0 \in \Sigma_\infty(F)$. 

Indeed, since $t^0 \in M_{\infty}(F)$, according to the Curve Selection Lemma, there exist analytic curves $\varphi(s)$ in $\mathbb{C}^n$ and  $\lambda(s)=(\lambda_1(s), \lambda_2(s), \ldots, \lambda_m(s))$ in $\mathbb{C}^m$, such that
\begin{itemize}
\item [(b1)] $\lim_{s \to 0}\|\varphi(s)\|= \infty,$
\item [(b2)] $\lim_{s \to 0}F(\varphi(s))= t^0,$
\item [(b3)] $\varphi(s)= \sum_{i=1}^m\lambda_i(s) \grad F_i(\varphi(s)).$
\end{itemize}

Set $I := \{i \ | \ \varphi_i \not\equiv 0\}.$ Because  (b1) then $I \ne \emptyset.$ For each $i
\in I$ we express
$$\varphi_i(s)= a_is^{\alpha_i} + \, \textit{terms of higher exponents},$$
where $a_i\neq 0, \alpha_i\in \Bbb{Z}$ and $\min_{i \in I} \alpha_i < 0.$ Similarly, let $J := \{j \ | \ \lambda_j \not\equiv 0\}.$ Because (b3) we have $J \ne \emptyset.$ 
For each $j
\in J,$ we also express
$$\lambda_j(s)= e_js^{\rho_j} + \, \textit{terms of higher exponents},$$
where $e_j \in \Bbb{C} \setminus \{0\}.$

Since $t^0_j\neq 0$ for all $j=1, \ldots, m$, it follows from (b2) that $F_j(\varphi(s))\neq 0$ for all $s$ small enough. Hence, the restriction of ${F_j}$ on $\Bbb{C}^I$ is non-trivial.  Therefore $\Gamma_{\infty}(F_j) \cap {\Bbb R}^I \ne \emptyset.$ (Here $\Bbb{C}^I:= \{x= (x_1, \ldots, x_n)\in \Bbb{C}^n \mid x_i=0, i\in I\}$; $\Bbb{R}^I$ is defined similarly). Let $d_j$ be the minimum of the linear function $P:= \sum_{i\in I}\beta_i \alpha_i$ on  $\Gamma_{-}(F_j) \cap {\Bbb R}^I.$ 
Let $\Delta_j:= \{\beta\in \Gamma_{-}(F_j) \cap {\Bbb R}^I\mid P(\beta)= d_j\}$. Then, for $i\in I$, we may rewrite (b3) as follows:
\begin{align}\label{dt(*)}
\sum_{j\in J}e_j\frac{\partial (F_j)_{\Delta_j}}{\partial x_i}(a)s^{d_j+\rho_j-\alpha_i} + \cdots = \overline{a_i}s^{\alpha_i} + \cdots,
\end{align}
where $a=(a_i)\in (\Bbb{C^*})^I$. Note that the set $\Delta_j\subset \Bbb{R}^I$ contains exponents $\beta= (\beta_1, \ldots, \beta_n)$ of monomials of $(F_j)_{\Delta_j}$ where $\beta_i=0$ for $i\notin I$. It follows that $(F_j)_{\Delta_j}$ does not depend on variables $x_i$ with $i\notin I$.

Denote
$$I^{'}= \{i \in I \mid \min_{j \in J}(d_j+ \rho_j - \alpha_i)= \alpha_i\}$$
and
$$J^{'}= \{j\in J \mid d_j + \rho_j= \min_{l\in J}(d_l+ \rho_l)\}.$$
We see that $i\notin I^{'}$ if and only if
\begin{align}
\sum_{j\in J^{'}}e_j\frac{\partial (F_j)_{\Delta_j}}{\partial x_i}(a)=0.\label{dangthuc3.5}
\end{align}
Moreover, if  $i\in I^{'}$ then
$$\sum_{j\in J^{'}}e_j\frac{\partial (F_j)_{\Delta_j}}{\partial x_i}(a)=\overline{a_i}.$$
We consider the following possibilities.

{\bf Case 1}: {\em The set $I'$ is non-empty.} Then, for all $j\in J$ we have
$$F_j(\varphi(s))= (F_j)_{\Delta_j}(a)s^{d_j}+ \, \textit{terms of higher exponents}.$$
If $(F_j)_{\Delta_j}(a)\neq 0$ then $d_j\geq 0$, unless $F_j(\varphi(s))\to \infty$ as $s\to 0$, contradicts to (b2). However, if  $d_j>0$ then $F_j(\varphi(s))\to 0$ and therefore $t^0_j=0= F_j(0, 0, \ldots, 0)$, contradicts to the hypothesis. Thus, we always have
\begin{align*}
d_j\cdot (F_j)_{\Delta_j}(a)=0.
\end{align*}
According to the Euler relation, we get
$$0= d_j\cdot (F_j)_{\Delta_j}(a)= \sum_{i\in I}\alpha_ia_i\frac{\partial (F_j)_{\Delta_j}}{\partial x_i}(a).$$
Then
\begin{align*}
0&= \sum_{j\in J^{'}}\sum_{i\in I}\alpha_i a_i e_j  \frac{\partial (F_j)_{\Delta_j}}{\partial x_i}(a) \\
  &= \sum_{i\in I^{'}, j\in J^{'}}\alpha_i a_i e_j \frac{\partial (F_j)_{\Delta_j}}{\partial x_i}(a) 
+ \sum_{i\notin I^{'}}\alpha_i a_i \left(\sum_{j\in J^{'}}e_j\frac{\partial (F_j)_{\Delta_j}}{\partial x_i}(a)\right).
\end{align*}
Consider this equation with (\ref{dangthuc3.5}) we obtain the following
$$\sum_{i\in I^{'}, j\in J^{'}}\alpha_i a_i e_j \frac{\partial (F_j)_{\Delta_j}}{\partial x_i}(a)=0.$$

 Furthermore, we observe that for all $i\in I^{'}$ then $\alpha_i = \min_{l=1, \ldots, m}\frac{d_l+\rho_l}{2}$ and according to (\ref{dt(*)}), for all $i\in I$ then  $\min_{l=1, \ldots, m}\frac{d_l+\rho_l}{2}\leq \alpha_i$.  Therefore, all numbers $\alpha_i$ with $i\in I^{'}$ are equal and non-zero. This, together with the above equation yields
$$\sum_{i\in I^{'}, j\in J^{'}} a_i e_j \frac{\partial (F_j)_{\Delta_j}}{\partial x_i}(a)=0.$$
It means, from (\ref{dt(*)}), that 
$$\sum_{i\in I^{'}}a_i \overline{a_i}=0.$$
This is impossible due to $a_i\neq 0$ for all $i\in I$.

{\bf Case 2}: {\em The set $I'$ is empty.} For all $i=1, \ldots, n$ we have
$$\sum_{j\in J^{'}}e_j\frac{\partial (F_j)_{\Delta_j}}{\partial x_i}(a)=0.$$
In other words
$$\sum_{j\in J^{'}}e_j\grad(F_j)_{\Delta_j}(a)=0.$$
Since $e_j\neq 0$ for all $j\in J^{'}$ the above equality implies
\begin{equation*}\label{bdtvehang}
\rank (J((F_1)_{\Delta_1}, (F_2)_{\Delta_2}, \ldots, (F_m)_{\Delta_m})(a))< m. 
\end{equation*}
Because $F$ is non-degenerate, by the definition of the Newton non-degeneracy, we have $0\in \Delta_j$  for every $j=1,\ldots, m.$, i.e., $\Delta_j$ is a bad face and $d_j=0$. As a consequence, we get
$$F_j(\varphi(s))= (F_j)_{\Delta_j}(a)+ \, \textit{terms of positive exponents}.$$
So $t^0_j= (F_j)_{\Delta_j}(a)$. It means $t^0\in \Sigma_\infty(F).$
\end{proof}

Denote by $\Sigma(F)$ the following set:
$$\Sigma(F):= K_0(F)\cup \Sigma_\infty(F) \cup \bigcup_{i=1}^m \{t=(t_1, \ldots, t_m) \in \Bbb{C}^m \mid t_i = F_i (0, 0, \ldots, 0)\}.$$

\begin{theorem}\label{theoremBifurcationNewton}
Let $F = (F_1, F_2, \ldots, F_m): \Bbb{C}^n \to  \Bbb{C}^m$ be a non-degenerate polynomial map. Then
$$B(F) \subseteq \Sigma(F).$$
\end{theorem}
\begin{proof}
The proof is straightforward from Theorem \ref{theoremBifurcationM-tame} and Theorem \ref{theoremM-tameNewton}.
\end{proof}

\begin{remark}{\rm 
(i) A version of Theorem \ref{theoremBifurcationNewton} was performed for mixed functions in \cite{CT}.

(ii) After this paper was accepted for publication, we found out that Theorem \ref{theoremBifurcationNewton} is proved in \cite{CDT} for more general cases.
}\end{remark}

\section{Case of one dimensional}
In this section, we consider the polynomial maps $F$ from $\Bbb{C}^n$ to $\Bbb{C}^{n-1}$ whose fibers are one complex dimensional. We show that for those maps, the problem of determining the bifurcation values is actually a topological problem. More precisely, we prove that if $F$ is a locally $C^0$-trivial fibration then it is a locally $C^{\infty}$-trivial fibration. We start with the following.

\begin{theorem}\label{thmtopo1}
Let $F: M\to N$ be a smooth map between smooth manifolds where $\dim_{\Bbb {R}} M = \dim_{\Bbb {R}} N +2.$ Let $t_0$ be a regular value of $F$. Assume that $F$ is a locally $C^0$-trivial fibration at $t_0$. Then $F$ is a locally $C^{\infty}$-trivial fibration at $t_0$.
\end{theorem}

Before proving the theorem, let us recall some definitions and results. A homotopy of a continuous map $h : X \to Y$ between topological spaces is a continuous map $H : X\times [0; 1] \to Y$, such that $H(x, 0) = h(x)$ for every $x\in X$.

Let $E, B$ be topological spaces and $\pi: E \to B$ be a continuous map.  We call  $\pi$ a {\it fibration} or equivalently, we say that it has the {\it homotopy lifting property}, if for every continuous map $h : X \to E$ whose source X is a polytope, every homotopy of $\pi \circ  h$ lifts to a homotopy of $h$.

The key result we use in this section is the following.

\begin{lemma}[See \cite{M}, Corollary 32]\label{lm2.4}
Let $\pi: E \to B$ be a (surjective, smooth) submersion-fibration, where $E$ and $B$ are smooth manifolds such that $\dim_{\Bbb{R}} E= \dim_{\Bbb{R}} B + 2$. Then $\pi$ is a locally $C^{\infty}$-trivial fibration.
\end{lemma}

\begin{proof}[Proof of Theorem \ref{thmtopo1}]
It follows from the hypothesis that there exist a neighborhood $D$ of $t_0$ and a homeomorphism ${\Phi}: F^{-1}(D)\rightarrow   F^{-1}(t_0)\times D$ such that the following diagram
\begin{displaymath}
\xymatrix{
F^{-1}(D) \ar[dr]^F  \ar[r]^{\Phi}  &  F^{-1}(t_0)\times D  \ar[d]^{pr_2}  \\
           &  D }
\end{displaymath}
 commutes.
 
 Since $\Phi$ is a homeomorphism and projections are fibrations then, the restriction $F_{|F^{-1}(D)}$ is also a fibration. According to Lemma \ref{lm2.4}, the restriction $F_{|F^{-1}(D)}$ is a $C^{\infty}$-trivial fibration.
\end{proof}

\begin{theorem}\label{mainthmst4}
Let $F: \Bbb{C}^n\to \Bbb{C}^{n-1}$ be a polynomial map and $t_0$ be a regular value of $F$. Then $t_0\notin B_{\infty}(F)$ if and only if $F$ is a locally $C^{0}$-trivial fibration at $t_0$.
\end{theorem}

\begin{proof}
 Since $t_0$ is a regular value of $F$ then there exists a neighbourhood $D$ of $t_0$ such that the restriction $F_{|F^{-1}(D)}$ is a submersion. The proof is then a consequence of Theorem \ref{thmtopo1}. 
\end{proof}

\begin{theorem}\label{thrm2.2}
Let $t_0\in N$ be a regular value of $F$. Then it is regular at infinity if and only if there exists a small ball $D$ centered at $t_0$, such that the inclusion of each fiber $F^{-1}(t)$ into $F^{-1}(D)$ is a weak homotopy equivalence, for all $t\in D$.
\end{theorem}

Before proving the theorem, we recall the followings.

\begin{definition}{\rm (\cite{M})
Two homotopies
$$H, H^{'}: X\times [0; 1] \to Y$$
are said to {\it have the same germ} if they coincide in a neighborhood of the subspace $X \times \{0\}$.
}
\end{definition}
\begin{definition}{\rm (\cite{M})
Let $\pi: E \to B$ be a continuous map, where $E, B$ are topological spaces. We call $\pi $ a {\it homotopic submersion}, or equivalently say that it has the
{\it germ-of-homotopy lifting property}, if for every continuous map $h : X \to E$ whose source $X$ is
a polytope, every germ-of-homotopy of $\pi \circ h$ lifts to a germ-of-homotopy of $h$.
}
\end{definition}

\begin{lemma}[See \cite{M}, Corollary 13]\label{lm2.5}
Let $\pi: E \to B$ be a surjective homotopic submersion. If the inclusion of each fiber into $E$ is a weak homotopy equivalence. Then $\pi$ is a fibration.
\end{lemma}

\begin{proof}[Proof of Theorem \ref{thrm2.2}]
If $t_0$ is regular at infinity then there is a small ball $D$ centered at $t_0$, such that the restriction $F_{|F^{-1}(D)}$ is trivial. It is easy to prove that the inclusion of each fiber $F^{-1}(t)$ into $F^{-1}(D)$ is a weak homotopy equivalence, for all $t\in D$. 

Let us assume that there is a ball $D$ centered at $t_0$, such that the inclusion of each fiber $F^{-1}(t)$ into $F^{-1}(D)$ is a weak homotopy equivalence for all $t\in D$ and the restriction $F_{|F^{-1}(D)}$ is surjective. It deduces from Lemma \ref{lm2.5} that $F: F^{-1}(D) \to D$ is a fibration. By applying Lemma \ref{lm2.4} we obtain that $F$ is differentially trivial at $t_0$.
\end{proof}

\section*{Acknowledgments} The author would like to thank Professor Pham Tien Son for useful discussions. The author would like to thank the referee(s) for carefully examining our paper and providing many valuable comments. The author is also thankful to the Laboratory Jean Alexandre Dieudonne for its hospitality. This research is funded by Vietnam National Foundation for Science and
Technology Development (NAFOSTED) under grant number 101.01-2011.44.

\end{document}